\documentclass{article}
\usepackage{maa-monthly}

\theoremstyle{theorem}
\newtheorem{theorem}{Theorem}
\newtheorem*{thm}{Theorem 1}

\newtheorem{prop}{Proposition}
\newtheorem{lem}{Lemma}
\newtheorem{conj}{Conjecture}
 
\theoremstyle{definition}

\newcommand{\R}{\mathbb{R}}
\newcommand{\E}{\mathbb{E}}

\newcommand{\al}{\alpha}
\newcommand{\be}{\beta}
\newcommand{\gam}{\gamma}

\newcommand{\sig}{\sigma}
\newcommand{\ep}{\epsilon}

\newcommand{\Om}{\Omega}

\renewcommand{\l}{\left}
\renewcommand{\r}{\right}
\newcommand{\half}{\frac{1}{2}}

\renewcommand{\c}[1]{\mathcal{#1}}
\renewcommand{\b}[1]{\mathbf{#1}}
\newcommand{\ol}[1]{\overline{#1}}

\newcommand{\rec}[1]{\frac{1}{#1}}
\newcommand{\f}[2]{\frac{#1}{#2}}
\newcommand{\floor}[1]{\l\lfloor #1\r\rfloor}
\newcommand{\ceil}[1]{\l\lceil #1\r\rceil}

\newcommand{\rand}[1]{\boldsymbol{#1}}
\renewcommand{\S}{\mathfrak{S}}

\usepackage{enumitem}

\setlist[itemize]{itemsep=2ex, topsep=2ex}

\title{Guessing about Guessing: Practical Strategies for Card Guessing with Feedback}
\author{Persi Diaconis \and Ron Graham \and Sam Spiro}
\date{\today}
\begin{document}
	\maketitle
\begin{abstract}
	In simple card games, cards are dealt one at a time and the player guesses each card sequentially.  We study problems where feedback (e.g., correct/incorrect) is given after each guess.  For decks with repeated values (as in blackjack where suits do not matter), the optimal strategy differs from the ``greedy strategy'' (of guessing a most likely card each round).  Further, both optimal and greedy strategies are far too complicated for real-time use by human players.  Our main results show that simple heuristics perform close to optimal.
\end{abstract}

\section{Introduction.}
Consider the following game.  A deck of $n$ cards labeled $1,2,\ldots,n$ is randomly shuffled.  A player then guesses each card (sequentially) as the cards are dealt face down on the table.  After each guess, the player is given some amount of feedback.  The main cases that we consider are the following:
\begin{itemize}[leftmargin=.3in]
	\item No feedback (the player is told nothing after each guess), 
	\item Complete feedback (the player is shown the value of the card after each guess),
	\item Yes/No feedback (the player is only told whether their guess was correct or not).
\end{itemize}
The final score is the total number of correct guesses made after all $n$ cards have been drawn.  Suppose that the player uses some fixed strategy $\c{G}$ for guessing cards based on the level of feedback they are given.   Let $\b{X}_i=1$ if the $i$th guess is correct (under the fixed strategy $\c{G}$) and $\b{X}_i=0$ otherwise.  Thus $\b{S}_n:=\b{X}_1+\cdots +\b{X}_n$ is the score in this game.  Then:
\begin{itemize}[leftmargin=.3in]
	\item With No feedback, $\Pr[\b{X}_i=1]=n^{-1}$ for all $1\le i\le n$, so the expected number of correct guesses is\[\E[\b{S}_n]=n\cdot n^{-1}=1.\]
	\item With Complete feedback, the maximizing strategy is to always guess a card known to be left in the deck.  Thus \[\Pr[\b{X}_1=1]=\rec{n},\ \Pr[\b{X}_2=1]=\rec{n-1},\ldots,\ \Pr[\b{X}_i=1]=\rec{n-i+1},\ldots,\]
	so $\E[\b{S}_n]=1+\half+\cdots +\rec{n}=\log n+O(1)$.
	\item With Yes/No feedback, consider the following strategy. Guess ``1.''  If this was correct, then the player knows ``1'' is no longer in the deck.  If incorrect, all the player knows for sure is that ``1'' is still in the deck.  Thus, heuristically, a good strategy would be to guess ``1'' until the player is told they are correct, then ``2'' until they are told they are correct (or until the deck runs out), then ``3,'' and so on.  And indeed, this strategy turns out to maximize $\E[\b{S}_{n}]$; see Diaconis and Graham~\cite{DG}.
	
	Under this strategy, the player always gets at least 1 correct guess, that is, $\Pr[\b{S}_n\ge 1]=1$.  There are at least 2 correct guesses if and only if ``2'' appears after ``1,'' so $\Pr[\b{S}_n\ge 2]=\half$.  More generally, for all $1\le k\le n$ we have $\Pr[\b{S}_n\ge k]=\rec{k!}$.  Thus under optimal play,
	\[\E[\b{S}_n]=\sum_{k=1}^n \Pr[\b{S}_n\ge k]=1+\half+\cdots+\rec{n!}=e-1+O(1/n!).\]
\end{itemize}

In real card games, cards have repeated values.  For example, in Blackjack or Baccarat, suits do not matter and all of the card types ``10, J, Q, K'' have value 10.  The analysis for No feedback in the repeated setting is trivial, but for Complete feedback things become more complex.  Under Complete feedback, one can easily prove that the strategy ``each round guess a card type that has the most number of copies left in the deck'' is an optimal strategy.  Thus computing the expected score under optimal play reduces to analyzing the performance of this (optimal) strategy, which was done by Diaconis and Graham~\cite{DG}.

The problem becomes significantly harder when we consider repeated values under Yes/No feedback.  In this setting, it is theoretically possible to compute the probability that the next card type is a given value ``$k$'' given the feedback from previous guesses; see Chung, Diaconis, Graham, and Mallows~\cite{CDGM} for information on how this can be done.  Thus one can define the  \textit{greedy strategy}: at each round, guess a card type that is most likely to appear next.  

The first issue with the greedy strategy is that it is very complicated to implement.  In particular, computing the probabilities needed in the strategy is equivalent to computing the permanents of certain matrices.  Thus this strategy is impractical for human players to use.

The second issue is that the greedy strategy is \textit{not} optimal; and in general it is not known what the optimal strategy is when one considers repeated values with Yes/No feedback. Even if such a strategy were known, it seems likely that it would be far too complex to implement by human players in practice. Given this, Diaconis and Graham~\cite{DG} posed the following two simple problems when the deck has $2n$ cards with two copies of each card type labeled $1,\ldots,n$ under Yes/No feedback:
\begin{itemize}[leftmargin=.3in]
	\item Is the optimal score of this game bounded as a function of $n$?
	\item Are there simple strategies in this game which perform reasonably well?
\end{itemize}
The answer to this first question has recently been answered positively by Diaconis, Graham, He, and Spiro~\cite{DGHS}.  In this article, we focus on the second question.  More precisely, we consider decks with $mn$ cards where each of the $n$ card types appear $m$ times.\footnote{A useful mnemonic is that $m$ is the \underline{m}ultiplicity of each card type, while $n$ is the \underline{n}umber of card types.}  We define $\b{S}_{m,n}$ to be the number of correct guesses made by the player with this deck if they use a given strategy under some level of feedback.  Our goal is to use practical strategies to bound the maximum and minimum possible values of $\E[\b{S}_{m,n}]$ under Yes/No feedback.  

To this end, we say that a strategy is an \textit{optimal strategy} if it achieves $\max \E[\b{S}_{m,n}]$, where the maximum ranges over all strategies, and similarly we say a strategy is an \textit{optimal mis\`ere strategy} if it achieves $\min \E[\b{S}_{m,n}]$.  Our main results are the following bounds on the score of the game under the optimal and optimal mis\`ere strategies.  Moreover, we will see that all of these bounds are achieved by relatively simple strategies.

\begin{theorem}\label{THM:LOWER}
	For all $m$ and $n\ge 8m$:
	\begin{itemize}[leftmargin=.3in]
		\item Under the optimal strategy with Yes/No feedback, \[\E[\b{S}_{m,n}]\ge m+\rec{40}\sqrt{m}.\]
		\item Under the optimal mis\`ere strategy with Yes/No feedback, \[\E[\b{S}_{m,n}]\le m-\rec{40}\sqrt{m}.\]
	\end{itemize}
\end{theorem}
It was proved in \cite{DGHS} that $\E[\b{S}_{m,n}]\le m+O(m^{3/4}\log m)$ uniformly in $n$ under the optimal strategy when $n$ is sufficiently large in terms of $m$, so one cannot prove a lower bound that is much stronger than Theorem~\ref{THM:LOWER} for $m$ large (though the constant $1/40$ can easily be improved).  

Our second main result consists of bounds for small (fixed values) of $m$ when $n$ is large.
\begin{theorem}\label{THM:PLUS}
	For $n$ sufficiently large, under the optimal strategy with Yes/No feedback,
	\[\E[\b{S}_{2,n}]\ge 2.91,\]
	\[\E[\b{S}_{3,n}]\ge 3.97.\]
\end{theorem}
Computational evidence suggests that the actual value of $\E[\b{S}_{2,n}]$ is close to the bound in Theorem~\ref{THM:PLUS}; see Table~\ref{tab:2Opt}.

Finally, we prove the following result regarding the optimal mis\`ere strategy.

\begin{theorem}\label{thm:minus}
	For all $m$ and $n$ sufficiently large in terms of $n$, under the optimal mis\`ere strategy with Yes/No feedback we have
	\[1-e^{-m}-o(1)\le \E[\b{S}_{m,n}]\le m-1+m^{-1}-m^{-1}e^{-m}+o(1).\]
\end{theorem}
When $m=1$, Theorem~\ref{thm:minus} implies a previous result of \cite{DG}, see Theorem~\ref{thm:1}.  We also note that for large $m$, the upper bound of Theorem~\ref{THM:LOWER} is more effective than that of Theorem~\ref{thm:minus}.
%We hope that our article contributes to a central problem of artificial intelligence: how to make black box alogorithms ``explainable'' or ``interpetable,'' see \cite{Machine1,Machine2}.

\subsection{Organization and notation.}  The rest of the article is organized as follows.  In Section~\ref{sec:history}, we discuss the history of the problem in more depth.   In Section~\ref{sec:prac}, we state some practical strategies that one can implement in the Yes/No feedback model and give some computational data for how they perform.  We then prove rigorous bounds on some of these strategies (or slight technical variants thereof) in order to prove our main results.  In particular, Theorem~\ref{THM:LOWER} is proved in Section~\ref{sec:lower} and Theorem~\ref{THM:PLUS} in Section~\ref{sec:plus}. Concluding remarks and further problems are given in Section~\ref{sec:concluding}.

We gather some notation that we use throughout the text.  We use standard asymptotic notation throughout this article.  In particular, we write $O(f(n))$ to refer to a function $g(n)$ such that $\limsup_{n\to \infty} \f{g(n)}{f(n)}<\infty$, and similarly $o(f(n))$ refers to a function $g(n)$ such that $\lim_{n\to \infty} \f{g(n)}{f(n)}=0$.  We write $\Om(f(n))$ if $\liminf_{n\to \infty} \f{g(n)}{f(n)}>0$ and $\Theta(f(n))$ to mean a function that is both $O(f(n))$ and $\Om(f(n))$.  We write $f\sim g$ if $\lim_{n\to \infty} \f{g(n)}{f(n)}=1$.  For more on asymptotic notation, see the book by Florescu and Spencer~\cite{S}.

We let $\S_{m,n}$ refer to the set of all words $\pi$ consisting of $m$ copies of each element in $[n]:=\{1,2,\ldots,n\}$, which we think of as a permutation of a deck of $mn$ cards with $n$ card types each occurring with multiplicity $m$.  We always write $\rand{\pi}$ to indicate an element of $\S_{m,n}$ chosen uniformly at random.  If $\c{G}$ is a deterministic strategy for the player under some level of feedback, we let $S(\c{G},\pi)$ be the score of the player if they follow strategy $\c{G}$ and the deck is shuffled according to $\pi$.  Thus the expected score following strategy $\c{G}$ is $\E[S(\c{G},\rand{\pi})]$.  As much as possible we denote random variables, such as $\b{S}_{m,n}$, in bold and strategies, such as $\c{G}$, in script.

\section{History.}\label{sec:history}
\subsection{Real-world applications.}
The original motivation for studying guessing problems that use sampling without replacement and different levels of feedback comes from certain real-world problems.

%Consider R.\ A.\ Fisher's ``Lady Tasting Tea'' experiment.  A lady claims she can tell if the tea in her drink was poured before or after the milk was poured.  To test the claim, Fisher proposed the following experiment.  Eight cups of tea are randomly placed with four poured before the milk and four after.  To help calibrate the lady's response, after each of her guesses, she is told ``correct'' or ``incorrect.''  If the lady guesses arbitrarily each round, then we expect 4 correct guesses.  However, if she uses the optimal strategy (assuming she has no ability to discern the teas), then she can expect to get around 5.3 correct guesses.  Conversely, if she's trying to show ``there's nothing to it,'' she might use the optimal mis\`ere strategy, giving 3.7 correct guesses in expectation.  

One such problem involves clinical trials, which was investigated by Blackwell and Hodges~\cite{BH} and Efron~\cite{Efron}.  Suppose that one wishes to test the effectiveness of two treatment options.  In comparing the effectiveness of two treatments on $2m$ patients, suppose it is decided that $m$ patients are to be randomly selected to receive each treatment.  Assume the patients arrive sequentially and that they must be ruled either ineligible or eligible before being assigned to one of the two treatments.  The order that the treatments will be administered is decided (randomly) in advance and is unknown to the physician.  However, a physician observing the outcome of each trial would know which treatment was most likely used, and thus they would know roughly the number of remaining treatments of each type at any given time.  This information might bias results if the physician ruled less healthy patients ineligible for trials where a favored treatment was less likely.  A natural measure of selection bias is the number of times the physician correctly guesses which treatment will be used next.  Blackwell and Hodges~\cite{BH} showed that $m+\half \sqrt{\pi m}-\half+O(1/m)$ correct guesses are made under optimal guessing.  Note that this setting corresponds to $n=2$ in our notation with either Complete or Yes/No feedback. (These two are the same when $n=2$.)

Independently, the same problem is studied by Ethier and Levin~\cite{EL}  as part of their work in evaluating card counting strategies in casino games such as Blackjack, Baccarat, and Trente et Quarante.  As cards are turned up during play, the deck changes composition,  so betting levels and actual strategy changes can make some games favorable. Exactly the same formula given by Blackwell and Hodges appears in Ethier~\cite[Problem 11.15]{Ethier}.  This wonderful book develops many further ideas tailored to these applications.

Another instance of this problem comes from testing if a subject has extrasensory perception (ESP).  A huge number of trials of the following experiment were performed by J.\ B.\ Rhine at the Durham parapsychology lab.  A deck of 25 cards with five copies of five different symbols ($0,\ +,\ \iiint,\ \square,\ *$) is shuffled, and a guessing subject guesses them sequentially as above.  It is common practice to give various kinds of feedback in these experiments, which is sometimes inadvertent if the guessing and sending subject are in the same room.  For extensive references, see Diaconis~\cite{D}.

%There are many other applications, see for example  Blackwell and Hodges~\cite{BH} and Efron~\cite{Efron} for applications to clinical trials, and Ethier and Levin~\cite{EL} for applications  to card counting strategies in casino games such as Blackjack, Baccarat, and Trente et Quarante.

\subsection{Some numbers.}
The main results in our article are asymptotic, giving results as $m$ and/or $n$ tend to infinity.  It is of interest to obtain exact results for decks of small size, both to get a feel for what the correct asymptotic bounds should be, and to make use of these results in the real-world applications mentioned above. 

For example, consider the classical ESP experiment which corresponds to $m=n=5$.  In \cite{DG}, a direct recursion was used to get the exact answers recorded in Table~\ref{tab:55}. As noted in \cite{D}, for actual ESP experiments, the highest recorded scores essentially never exceeded 8.65, and the highest scores among long-term subjects never exceeded 6.63; so the numbers of Table~\ref{tab:55} help benchmark high scoring experiments.  Much more can be said about this well-studied case; see Diaconis, Gatto, and Graham~\cite{DGG} and Gatto~\cite{G}.

\begin{table}
	\caption{Expected number of correct guesses under an optimal strategy with Yes/No feedback when $m=n=5$.  The results are rounded to two decimal places.}
	\label{tab:55}
	\begin{center}
		\begin{tabular}{ | c | c | c |} 
		\hline
		No feedback & Yes/No feedback & Complete feedback\\ 
		\hline 
		5 & 6.65 & 8.65\\ 
		\hline
		\end{tabular}
	\end{center}
	
\end{table}

Table~\ref{tab:2Various} gives some numbers for $m=2$ and small $n$ under Yes/No feedback using various strategies.  This data was obtained by implementing the algorithms detailed in Gatto~\cite{G}.  The \textit{linear strategy} mentioned in this table is the following: assume at some stage one has incorrectly guessed each $i$ a total of $g_i$ times and that one has correctly guessed each $i$ a total of $n-c_i$ times (so there are $c_i$ copies of $i$ that have not been confirmed which may or may not be among the remaining cards).  Given this, we guess the type $i^*$ such that $c_i+.51g_i$ is maximized (breaking ties at random).  The value .51 was found by experimenting with different choices of parameters. Table~\ref{tab:2Various} shows that, in addition to the greedy strategy being nonoptimal, it is seemingly outperformed by this simple linear strategy. Observe by Table~\ref{tab:2Various} that the greedy and linear strategies are close to optimal for small $n$.  While we suspect that this phenomenon holds for all $n$, we do not know how to prove this.

%It is natural to ask ``why'' the greedy strategy is not optimal with Yes/No feedback.  One heuristic is that there is a tradeoff between score and information: sometimes it is better to guess a card which is less likely if knowledge about this card type is more valuable.  However, the real answer to this question is ``we don't know!''

We find it rather surprising that the simple linear ``card counting'' strategy performs so well.  In the classical ESP setting of $m=n=5$, a similar strategy can be used where one guesses the $i^*$ that maximizes $c_i+.35g_i$.  This simple strategy has expectation 6.6149, which is better than the greedy strategy!  Here .35 is again the result of a computer search, and in general we expect for linear strategies of the form $c_i+\be g_i$ that the optimal value of $\be$ should decrease as $n$ increases.

\begin{table}
	\caption{Expected number of correct guesses with Yes/No feedback under various strategies when $m=2$. The results are rounded to four decimal places.}
	\label{tab:2Various}
	\begin{center}
		\begin{tabular}{ | c | c | c | c|} 
			\hline
			& Optimal  & Greedy & Linear\\ 
			\hline
			$\E[\b{S}_{2,2}]$ & 2.8333 & 2.8333 & 2.8333\\ 
			\hline
			$\E[\b{S}_{2,3}]$ & 3.0111 & 3.0111 & 3.0111\\ 
			\hline 
			$\E[\b{S}_{2,4}]$ & 3.0452 & 3.0333 & 3.0433\\ 
			\hline
			$\E[\b{S}_{2,5}]$ & 3.0467 & 3.0222 & 3.0441\\ 
			\hline 
		\end{tabular}
	\end{center}
	
\end{table}

At a first glance of Table~\ref{tab:2Various}, it may appear as if $\E[\b{S}_{2,n}]$ under the optimal strategy is nondecreasing, but Table~\ref{tab:2Opt} shows that this is not the case.  Note that all of these values seem to be very close to 3.  It is unclear what the true asymptotic value should be, though Theorem~\ref{THM:PLUS} shows that it is as least $2.91$.

\begin{table}
	\caption{Expected number of correct guesses with Yes/No feedback  under an optimal strategy when $m=2$. The results are rounded to four decimal places.}
	\label{tab:2Opt}
	\begin{center}
		\begin{tabular}{ | c | c | c | c| c| c|} 
			\hline
			$n$ & 6 & 7 & 8 & 9 & 10\\ 
			\hline
			$\E[\b{S}_{2,n}]$ & 3.0376 & 3.0323 & 3.0260 & 3.0219 & 3.0186 \\
			\hline 
		\end{tabular}
	\end{center}
	
\end{table}

\subsection{Previous research.}
This article continues research in \cite{DG,DGHS}. In these papers, the problem of determining $\E[\b{S}_{m,n}]$ under both optimal and optimal mis\`ere strategies with Complete feedback was essentially solved.  The authors were able to prove this result in large part due to the optimal and optimal mis\`ere strategies being known when Complete feedback is given.  

Despite not knowing the optimal strategy under Yes/No feedback in general, it was proved in \cite{DGHS} that $\E[\b{S}_{m,n}]$ is at most $m$ plus lower-order terms under the optimal strategy.  More precisely, the following was proved.
\begin{theorem}[\cite{DGHS}, Theorem 1.3]\label{thm:partAsy}
	If $n$ is sufficiently large in terms of $m\ge 2$, then under an optimal strategy with Yes/No feedback,
	\[\E[\b{S}_{m,n}]=m+O(m^{3/4}\log m),\]
	where this bound holds uniformly in $n$.
\end{theorem}
This shows that the lower bound of Theorem~\ref{THM:LOWER} is close to best possible.  It also shows that for fixed $m$, under an optimal strategy with Yes/No feedback, $\E[\b{S}_{m,n}]$ is bounded as a function of $n$.  This is in sharp contrast to the Complete feedback model, where one can guarantee $\Om(\log n)$ correct guesses in expectation; see \cite{DGHS}.

Theorem~\ref{thm:partAsy} does not given an effective bound on $\E[\b{S}_{m,n}]$ under Yes/No feedback for any fixed $m$.  For example, it does not tell us whether $\E[\b{S}_{2,n}]\le 100$ for $n$ sufficiently large, and it is natural to ask for bounds when $m$ is a fixed small integer. The case $m=1$ was completely solved in \cite{DG}.

\begin{theorem}[\cite{DG}, Theorems 5 and 6]\label{thm:1}
	Under the optimal strategy with Yes/No feedback,
	\[\E[\b{S}_{1,n}]=e-1+O(1/n!)\approx 1.72,\]
	and under the optimal mis\`ere strategy,
	\[\E[\b{S}_{1,n}]=1-e^{-1}+O(1/n!)\approx .632.\]
\end{theorem}
The proof of Theorem~\ref{thm:1} relies on an optimal strategy which is easy to analyze when $m=1$, and without this it seems difficult to nail down asymptotic values of $\E[\b{S}_{m,n}]$ in general.

%To analyze $\b{S}_{m,n}$ under Yes/No feedback, it is often useful to determine the probability that the next card is of type $i$ given the past history of feedback.  In Chung, Diaconis, Graham, and Mallows~\cite{CDGM}, this problem is reduced to evaluating the permanent of certain zero/one matrices.  While general evaluation of permanents is \#P-complete, the matrices arising here have enough structure to permit evaluation in practical problems.  The study of permanents in relation to card guessing problems is further studied in \cite{CDG, DGH, DK}. 

%Lastly, we note that in real life applications, the deck is not always shuffled uniformly at random, and it is natural to consider what happens with other common forms of shuffling.    The case where dovetail shuffles are used is studied by Ciucu~\cite{C}, and recently Liu~\cite{L} investigated the case where riffle shuffles are used.  We emphasize that for our results, we only consider decks which are shuffled uniformly at random.

\section{Practical Strategies and Computational Data.}\label{sec:prac}
Perhaps the simplest strategy to use in the Yes/No feedback model is the strategy where one guesses ``1'' every round.  This strategy always gives exactly $m$ correct guesses.  Of course, this strategy is somewhat silly.  One can always do at least as well by using the strategy: guess card ``1'' until the player is told $m$ guesses were correct, then guess card ``2'' until the player is told $m$ guesses were correct, and so on.  We call this the \textit{safe strategy} and denote it by $\c{S}$.  While this does better than the trivial strategy of guessing ``1'' every round, heuristic reasoning suggests that it does not do much better when $m$ is large.  

Indeed, the probability that the safe strategy gives at least $m+k$ correct guesses for $k\le m$ is exactly ${2m-k\choose m}/{2m\choose m}$ (since this happens whenever the $m$ ``1''s appear in the first $2m-k$ positions in the deck).  Thus the expected number of correct guesses of ``1'' and ``2'' card types using the safe strategy is 
\begin{align*}\sum_{k=0}^{2m}\Pr[\b{S}_{m,n}\ge k]&=m-1+\sum_{k=0}^m \f{{2m-k\choose m}}{{2m\choose m}}\\&=m-1+\f{{2m+1\choose m+1}}{{2m\choose m}}\\&=m-1+\f{2m+1}{m+1}=m+1-\rec{m+1}.\end{align*}
For $n>2$ the probability of guessing any ``3'' correctly is less than ${2m\choose m}^{-1}$, so for large $m$ the safe strategy will typically only guess ``1'' and ``2.''  Thus we do not expect much more than $m+1$ correct guesses using the safe strategy.

Another natural strategy is the \textit{shifting strategy} $\c{F}$ where the player guesses ``1'' until they get a correct guess, then ``2'' until they get a correct guess, and so on; and upon guessing an ``$n$'' correctly, they go back to guessing ``1''s, and then ``2''s, and so on. Observe that the safe and shifting strategies look identical until the first ``1'' is guessed correctly.  Because of this, the player can choose which strategy they wish to use after seeing the first ``1,'' where intuitively they should use the shifting strategy if the first ``1'' shows up early and the safe strategy otherwise.  

To this end, we define the \textit{$\gam$-shifting strategy} $\c{F}_\gam$ by guessing ``1'' until a correct guess is made at time $t$.  If $t\ge \gam mn$, the player proceeds as in the safe strategy, and otherwise they proceed as in the shifting strategy.  For example, $\c{F}_0$ is the safe strategy and $\c{F}_1$ is the shifting strategy, so $\c{F}_\gam$ serves as a sort of interpolation between these two strategies.

%Of course one can generalize this idea even further.  For example, one can define the \textit{$(\gam,\gam')$-shifting strategy} $\c{F}_{\gam,\gam'}$ by guessing ``1'' until a correct guess is made at time $t$.  If $t\ge \gam mn$ then one plays as in the safe strategy, otherwise one starts playing as in the shifting strategy until a ``2'' is guessed correctly at time $t'$.  If $t'\ge \gam'mn$, then one essentially switches to the safe strategy (continuing to guess ``2''s until all of them are guessed then ``3''s and so on), and otherwise one keeps using the shifting strategy. For example, $\c{F}_{\gam,1}$ is the same as $\c{F}_\gam$. One can continue to build upon this idea with an arbitrary amount of complexity.

We next define the \textit{halfway strategy} $\c{H}^+$ by guessing ``1'' for the first $\half mn$ trials.  If at most $\half m$ cards have been guessed correctly, then keep guessing ``1'' (guaranteeing $m$ points at the end of the game), and otherwise guess ``2'' for the rest of the game. The intuition is that if there are many copies of ``1'' in the first half of the deck, then there will be slightly more copies of ``2'' (or any other card type) in the second half of the deck.  This intuition is made rigorous in Lemma~\ref{L-K2}.

We now turn to strategies which get few correct guesses.  Analogous to $\c{H}^+$, we define the \textit{halfway strategy} $\c{H}^-$ by guessing ``1'' for the first half of the game, then to keep guessing ``1''s if more than $\half m$ correct guesses were made, and otherwise guessing ``2''s the rest of the game.  Finally, we define the \textit{avoiding strategy} $\c{A}$ by guessing ``1,'' then ``2,'' then ``3,'' and so on until some ``$k$'' is guessed correctly, at which point one guesses ``$k$'' for the rest of the game.  If the player does not guess any of the first $n$ cards correctly, then they again guess ``1,'' then ``2,'' and so on until a card ``$k$'' is guessed correctly, and then they guess ``$k$'' for the rest of the game.

In Table~\ref{tab:strats}, we present computational data for most of these strategies.  The first two rows are the exact value of $\E[\b{S}_{m,n}]$ using the stated strategy with Yes/No feedback.  The next two rows represent the sample mean of $S(\c{G},\rand{\pi})$ obtained after sampling $t=10^6$ permutations of $\S_{m,n}$.  The entries corresponding to $\c{F}_\gam$ indicate which value of $\gam$ was used, and these values were chosen to roughly maximize the expectation.  Each entry is rounded after three decimal places.
 \begin{table}
 \caption{(Simulated) values of $\E[\b{S}_{m,n}]$ under various strategies.}
\label{tab:strats}
\begin{center}
	%\captionof{table}{Graphs Satisfying $X^r=f(Y)$}
	\begin{tabular}{ | c | c | c | c | c| c|} 
		\hline
		& $\c{S}$ &  $\c{F}$ & $\c{F}_\gam$ & $\c{H}^+$& $\c{H}^-$\\ 
		\hline
		$m=2,n=6$ (exact) & 2.737 & 2.751 & $\c{F}_{.3}$: 2.941 & 2.212& 1.682\\ 
		\hline
		$m=3,n=5$ (exact)& 3.772 & 3.753 & $\c{F}_{.25}$: 4.006 & 3.431& 2.569\\ 
		\hline
		$m=4,n=10\ (t=10^6)$ & 4.806 & 4.831 & $\c{F}_{.2}$: 5.100 &  4.370 & 3.585\\ 
		\hline 
		$m=5,n=20\ (t=10^6)$ & 5.835  & 5.856 & $\c{F}_{.15}$: 6.136 & 5.482 & 4.516\\ 
		\hline 
	\end{tabular}
\end{center}

\end{table}

%We note that for $m=2$, all of these strategies perform worse than the simple linear strategy defined prior to Table~\ref{tab:2Various}.  
The main benefit of the strategies of this section is that we can give rigorous bounds on their performance (or more precisely, on technical variants of the strategies which are easier to analyze).  For example, we provided no  computational data for the avoiding strategy, since it turns out that we can determine its expectation asymptotically for all $m$.

\begin{prop}\label{prop:A}
	For any fixed $m$, using the avoiding strategy with Yes/No feedback gives
	\[\E[\b{S}_{m,n}]\sim m-1+m^{-1}-m^{-1}e^{-m}.\]
\end{prop}
\section{Proof of Theorem~\ref{THM:LOWER}: the Halfway Strategies $\c{H}^\pm$.}\label{sec:lower}
In this section, we prove Theorem~\ref{THM:LOWER} using a technical variant of the halfway strategies $\c{H}^\pm$, which we denote by $\c{H}^\pm_*$ and define as follows.  Guess ``1'' a total of $\floor{mn/2}$ times.  If at most most $\half m+\half \sqrt{m}$ correct guesses have been made, then we continue to guess ``1'' for the rest of the game, otherwise we guess ``2'' for the rest of the game otherwise.  Similarly, $\c{H}_*^-$ is defined by guessing ``1'' a total of $\floor{mn/2}$ times, then continuing if at most $\half m-\half \sqrt{m}$ correct guesses have been made, and with ``2'' being guessed otherwise.

To show that $\c{H}^+_*$ gives the correct lower bound, we first show that the probability of getting at least $\half m+\half \sqrt{m}$ correct guesses in the first phase is relatively large.

\begin{lem}\label{L-K1}
	For $\pi \in \S_{m,n}$, let $K_1(\pi)$ denote the number of $t\le \floor{mn/2}$ with $\pi_t=1$.  Then for $m\ge 2$ and $n\ge 8 m$ and all $0\le k\le m$, we have 
	\begin{align*} \Pr[K_1(\rand{\pi})=k]\ge \half \cdot 2^{-m}{m\choose k}.\end{align*}
\end{lem}
That is, the probability of having exactly $k$ of the $m$ ``1''s appearing in the first half of $\rand{\pi}$ is roughly the probability of having $k$ heads in a series of $m$ coin tosses.

\begin{proof}
	First assume $mn$ is even.  Then \begin{equation}\Pr[K_1(\rand{\pi})=k]={mn/2\choose k}{mn/2\choose m-k}/{mn\choose m}.\label{E-K1}\end{equation}  We recall the bounds ${N\choose r}\le N^r/r!$ and \[{N\choose r}\ge \f{(N-r)^r}{r!}=(1-r/N)^r \f{N^r}{r!}\ge (1-r^2/N) \f{N^r}{r!},\] where the last inequality uses $(1+x)^r\ge 1+rx$, which holds for $x\ge -1$ and $r\ge 1$.  Using these bounds and \eqref{E-K1}, we find for $0\le k\le m$ and $n\ge 8m$ that \begin{align*}\Pr[K_1(\rand{\pi})=k]&\ge \f{(mn/2)^m m!}{(mn)^m k!(m-k)!}(1-2k^2/mn)(1-2(m-k)^2/mn)\\&\ge 2^{-m} {m\choose k}\cdot (1-2m/n)^2 \ge 2^{-m} {m\choose k}\cdot \half.\end{align*}
	
	This completes the proof for $mn$ even.  If $mn$ is odd, then a similar analysis as before gives
	 \begin{align*}\Pr[K(\rand{\pi})=k]&\ge 2^{-m}{m\choose k}\cdot \f{(mn-1)^k(mn+1)^{m-k}}{(mn)^m}(1-2m^2/(mn-1))^2\\ &\ge 2^{-m}{m\choose k}\cdot (1-1/mn)^m(1-2m^2/(mn-1))^2.\end{align*}
	Using $(1+x/r)^r\ge 1+x$, $n\ge 8m$, as well as $mn$ odd and $m\ge 2$ implying $m\ge 3$,  we find that the above quantity is at least
	\[2^{-m}{m\choose k}\cdot (1-1/8m)(1-2m^2/(8m^2-1))^2\ge 2^{-m}{m\choose k}\cdot (23/24)(53/71)^2,\]
	which is greater than $2^{-m}{m\choose k}\cdot \half$ as desired.
\end{proof}
%\ge 2^{-m} {m\choose k}(1-4m/n)^3
We show that the bound in Lemma~\ref{L-K1} is large for $k\ge \half m+\half \sqrt{m}$ by using the following anti-concentration result for binomial random variables.

\begin{lem}\label{lem:CLT}
	Let $\b{B}_m$ denote a binomial random variable with $m$ trials and probability of success $1/2$.  Then for all $m$, 
	\[\Pr\l[\b{B}_m\ge \half m+\half \sqrt{m}\r]\ge \f{7}{64}.\]
\end{lem}
We note that this bound is best possible by considering $m=6$.
\begin{proof}
	By the symmetry of the binomial distribution, it suffices to prove
	\[\Pr\l[\b{B}_m\le \half m-\half \sqrt{m}\r]\ge \f{7}{64}.\]
	By \cite[Corollary 1.2]{HM}, for all $m$ we have \[\l|\Pr\l[\b{B}_m\le \half m-\half \sqrt{m}\r]-\Phi(-1)\r|\le \rec{\sqrt{2\pi m}},\]
	where $\Phi(x)=\rec{\sqrt{2\pi}} \int_{-\infty}^x e^{-t^2/2}dt$ is the cumulative distribution function of a standard normal distribution.  This implies
	\[\Pr\l[\b{B}_m\le \half m-\half \sqrt{m}\r]\ge .158+\rec{\sqrt{2\pi m}},\] 
	which gives the desired bound for $m\ge 68$, and one can verify the result for smaller $m$ by aid of a computer.
\end{proof}

Finally, we show that conditional on getting at least $\half m+\half \sqrt{m}$ correct guesses in the first phase of the strategy, one expects to guess ``2'' correctly at least $\half m$ times in the second phase.
\begin{lem}\label{L-K2}
	For $\pi \in \S_{m,n}$, let $K_2(\pi)$ denote the number of $t>\floor{mn/2}$ with $\pi_t=2$, and define $K_1(\pi)$ as in Lemma~\ref{L-K1}.  If $k\ge \half m$, then \[\E[K_2(\rand{\pi})|K_1(\rand{\pi})=k]\ge \half m,\]
	and if $k< \half m$,
	\[\E[K_2(\rand{\pi})|K_1(\rand{\pi})=k]\le \half m.\]
\end{lem}
\begin{proof}
    For $i\ge 2$ define $K_i(\pi)$ to be the number of $t>\floor{mn/2}$ with $\pi_t=i$.  Observe that for all $k$ and $i\ge 2$, we have $\E[K_2(\rand{\pi})|K_1(\rand{\pi})=k]=\E[K_i(\rand{\pi})|K_1(\rand{\pi})=k]$.  Thus
    \begin{align*}(n-1)\cdot \E[K_2(\rand{\pi})|K_1(\rand{\pi})=k]&=\E\big[\sum_{i\ge 2} K_i(\rand{\pi})|K_1(\rand{\pi})=k\big]\\&=mn-\floor{mn/2}-m+k,\end{align*}
    where this last step used that the expectation is exactly the number of cards among the last $mn-\floor{mn/2}$ that are not of type ``1,'' which is deterministically equal to $mn-\floor{mn/2}-m+k$.  
    
    With this, if $k\ge \half m$ we find \[\E[K_2(\rand{\pi})|K_1(\rand{\pi})=k]\ge \rec{n-1}(mn-mn/2-m/2)=\half m,\]
	and similarly if $k<\half m$, then $k$ being an integer implies $k\le \half m-\half$, and we have
	\[\E[K_2(\rand{\pi})|K_1(\rand{\pi})=k]\le \rec{n-1}(mn-(mn-1)/2-m/2-1/2)=\half m,\]
	proving the result.

\end{proof}

With all this we can prove Theorem~\ref{THM:LOWER}, which for the convenience of the reader we restate here.
\begin{thm}
	For all $m$ and $n\ge 8m$:
	\begin{itemize}[leftmargin=.3in]
		\item Under the optimal strategy with Yes/No feedback, \[\E[\b{S}_{m,n}]\ge m+\rec{40}\sqrt{m}.\]
		\item Under the optimal mis\`ere strategy with Yes/No feedback, \[\E[\b{S}_{m,n}]\le m-\rec{40}\sqrt{m}.\]
	\end{itemize}
\end{thm}
\begin{proof}
	If $m=1$, then this follows from Theorem~\ref{thm:1}, so from now on we assume $m\ge 2$.  We first prove our lower bound on the optimal strategy by considering the aforementioned strategy $\c{H}^+_*$  of guessing ``1'' a total of $\floor{mn/2}$ times, then guessing ``1'' the rest of the game if we guessed fewer than  $\half m+\sqrt{m}$ cards correctly, and otherwise guessing ``2'' for the rest of the game.  
	
	For ease of notation, we let $\b{K}_i=K_i(\rand{\pi})$ for $i=1,2$ as defined in Lemmas~\ref{L-K1} and \ref{L-K2}.  With this, we see that $S(\c{H}_*^+,\rand{\pi})=m$ if $\b{K}_1<\half m+\half\sqrt{m}$ and $S(\c{H}_*^+,\rand{\pi})=\b{K}_1+\b{K}_2$ otherwise.    Thus $\E[S(\c{H}_*^+,\rand{\pi})]$ is equal to
	{\small \begin{align*}Pr\Big[\b{K}_1&<\half m+\half \sqrt{m}\Big]\cdot m+\sum_{k\ge \half m+\half \sqrt{m}} \Pr[\b{K}_1=k]\cdot (k+\E[\b{K}_2|\b{K}_1=k])\\ &\ge \hspace{.5em}\Pr\Big[\b{K}_1<\half m+\half \sqrt{m}\Big]\cdot m+\sum_{k\ge \half m+\half \sqrt{m}}\Pr[\b{K}_1=k]\cdot\l(\half m+\half \sqrt{m}+\half m\r)\\ 
		&=  \hspace{.5em}m+\Pr\Big[\b{K}_1\ge \half m+\half\sqrt{m}\Big]\cdot \half \sqrt{m},\end{align*}}
	where the inequality used Lemma~\ref{L-K2}. By Lemmas~\ref{L-K1} and \ref{lem:CLT}, we have \[\Pr\l[\b{K}_1\ge \half m+\half \sqrt{m}\r]\ge \half \Pr\l[\b{B}_m\ge \half  m+\half \sqrt{m}\r]\ge \rec{20},\]
	and with this we conclude the desired lower bound for the optimal strategy.
	
	For the optimal mis\`ere strategy, essentially the same analysis as above applies to $\c{H}_*^-$, the only significant change being that we use the second half of Lemma~\ref{L-K2} instead of the first half.  We omit the details.
\end{proof}

\section{Proof of Theorem~\ref{THM:PLUS}: the $\gam$-shifting Strategy $\c{F}_\gam$.}\label{sec:plus}
Intuitively we use the $\gam$-shifting strategy $\c{F}_\gam$ to achieve the lower bound of Theorem~\ref{THM:PLUS} for $\E[\b{S}_{2,n}]$ and $\E[\b{S}_{3,n}]$, though for technical reasons it will be convenient to use a variant that never guesses card types larger than some cutoff value $k$.  To aid in our proof, we use the following lemma for approximating sums with integrals.

\begin{lem}[\cite{S}, Theorem 4.2]\label{lem:integral}
	Let $a<b$ be integers.  Let $h$ be an integrable function on $[a-1,b+1]$, $S=\sum_{i=a}^b h(i)$, and $I=\int_a^b h(x)\,dx$.  Let $M$ be such that $|h(x)|\le M$ for all $a-1\le x\le b+1$.  Suppose $[a-1,b+1]$ can be broken up into at most $r$ intervals such that $h$ is monotone on each.  Then \[|S-I|\le 6rM.\]
\end{lem}
We now prove Theorem~\ref{THM:PLUS}, which for the convenience of the reader we restate here.
\begin{thm}
	For $n$ sufficiently large, under the optimal strategy with Yes/No feedback,
	\[\E[\b{S}_{2,n}]\ge 2.91,\]
	\[\E[\b{S}_{3,n}]\ge 3.97.\]
\end{thm}
\begin{proof}
	The argument is somewhat complex, so to start we consider $m=2$. Fix some integer $k$ and real $\gam$, and define the $(k,\gam)$-shifting strategy $\c{F}_{k,\gam}$ as follows. The player guesses ``1'' until they get a correct guess.  If this happens in fewer than $2\gam n$ guesses, they use a modified shifting strategy $\c{F}_k$ where they guess ``2'' until they get a correct guess, then ``3'' until they get a correct guess, all the way up to ``$k$,'' after which they go back to guessing ``1'' until they get a correct guess, then ``2,'' and so on.  If they guess ``$k$'' correctly a second time, they keep guessing ``$k$'' (effectively giving up on trying to get more points).  If the first ``1'' appears after $2\gam n$ guesses, they use a modified safe strategy $\c{S}_k$ where they guess ``1'' until they guess both correctly, then ``2'' until they guess both correctly, and so on until they guess both ``$k$''s correctly, after which they keep guessing ``$k$'' for the rest of the game.
	
	Given a permutation $\pi$, let $F_k(\pi)$ be the score the player gets using the modified shifting strategy $\c{F}_k$ on $\pi$, and let $G_k(\pi)$ be the score they get using the modified safe strategy $\c{S}_k$ on $\pi$.  Let $\pi_1^{-1}$ denote the index of the leftmost ``1'' in $\pi$.  For example, if $k=3$ and $\pi=41234132$, then $\pi_1^{-1}=2,\ F_3(\pi)=5$ (coming from guessing 1, 2, 3, 1, 2 correctly) and $G_3(\pi)=3$ (coming from guessing 1, 1, 2 correctly).  As another set of examples,	$\{(\pi,\pi_1^{-1},F_2(\pi),G_2(\pi)):\pi \in \S_{2,2}\}$ is equal to
	{\small \begin{align}
			\{(1122,1,2,4),\ (1212,1,4,3),\ (1221,1,3,2),\ (2112,2,2,3),\ (2121,2,3,2),\ (2211,3,1,2)\}.\label{E-Ex}
	\end{align}}
	
	Let $\pi^{\le k}$ denote $\pi$ after deleting every letter larger than $k$.  By construction,  $S(\c{F}_{k,\gam},\pi)=F_k(\pi)=F_k(\pi^{\le k})$ if $\pi_1^{-1}<2\gam n$ and $S(\c{F}_{k,\gam},\pi)=G_k(\pi^{\le k})$ otherwise.  That is, $S(\c{F}_{k,\gam},\pi)$ depends entirely on $\pi^{\le k}$ and $\pi_1^{-1}$.  In particular,
	{\begin{align}\E[S(\c{F}_{k,\gam},\rand{\pi})]= &\sum_{t<2\gam n}\ \sum_{\sig\in \S_{2,k}}F_k(\sig)\Pr[\rand{\pi}_1^{-1}=t\ \cap\ \rand{\pi}^{\le k}=\sig]\nonumber \\&+\sum_{t\ge 2\gam n}\ \sum_{\sig\in \S_{2,k}}G_k(\sig)\Pr[\rand{\pi}_1^{-1}=t\ \cap\ \rand{\pi}^{\le k}=\sig].\label{E-Gam}\end{align}}
	
	With this in mind, we wish to determine $\Pr[\rand{\pi}_1^{-1}=t\ \cap\ \rand{\pi}^{\le k}=\sig]$.  Fix some $\sig\in \S_{2,k}$ and define $q=\sig_1^{-1}-1$.  We claim that the total number of $\pi\in \S_{2,n}$ with $\pi_1^{-1}=t$ and $\pi^{\le k}=\sig$ is \begin{equation}\label{E-qCount}{t-1\choose q}{2n-t\choose 2k-1-q}\f{(2n-2k)!}{ 2^{n-k}}.\end{equation}   Indeed, first we choose the indices where the letters of type $\{1,\ldots,k\}$ will go as follows.  The leftmost ``1'' must go in position $t$, and then one chooses the $q$ positions to the left of $t$ and the $2k-1-q$ positions to the right of $t$ where the remainder of these symbols go in ${t-1\choose q}{2n-t\choose 2k-1-q}$ ways.  Once these indices are chosen, the relative order of the symbols in $\{2,\ldots,k\}$ is determined by $\sig$.  One then arranges the remaining symbols that are not in $\{1,\ldots,k\}$ in $(2n-2k)!/2^{n-k}$ total ways.  This proves the claim.  We note that \eqref{E-qCount} can be written as
	\begin{align*}\f{(t-1)(t-2)\cdots(t-q)\cdot (2n-t)\cdots(2n-t-2k+q+2)\cdot (2n-2k)!}{2^{n-k}q!(2k-1-q)!}&\\\ge \f{(t-2k)^q(2n-t-2k)^{2k-1-q}(2n-2k)!}{2^{n-k}q!(2k-1-q)!}&,\end{align*}
	where we used that $0\le q<2k$.  By dividing this by $|\S_{2,n}|=(2n)!/2^n$, we conclude that
	\begin{align}
		\Pr[\rand{\pi}_1^{-1}=t\ \cap\ \rand{\pi}^{\le k}=\sig]&\ge\f{2^{k}}{q!(2k-1-q)!}\f{(t-2k)^q(2n-t-2k)^{2k-1-q}}{(2n)\cdots(2n-2k+1)}\nonumber\\  &\ge\f{2^{k}}{q!(2k-1-q)!}\f{(t-2k)^q(2n-t-2k)^{2k-1-q}}{(2n)^{2k}}.\label{E-Prob}
	\end{align}
	Let \[\phi_{q,n}(t):=\f{2^k}{2n\cdot q!(2k-1-q)!}\l(\f{t}{2n}\r)^q\l(1-\f{t}{2n}\r)^{2k-1-q},\]  
	which is simply \eqref{E-Prob} after replacing each $t-2k$ term with $t$, and intuitively this should be close to \eqref{E-Prob} because $k$ is a fixed integer.  More precisely, we claim that \eqref{E-Prob} is equal to $\phi_{q,n}(t)+O(n^{-2})$ for $1\le t\le 2n$, which implies \begin{equation}
		\Pr[\rand{\pi}_1^{-1}=t\ \cap\ \rand{\pi}^{\le k}=\sig]\ge \phi_{q,n}(t)+O(n^{-2}).\label{E-Phi}
	\end{equation}  Indeed, consider the numerator of the second fraction in \eqref{E-Prob} as a polynomial in $2k$.  Because $t\le 2n$, for any integer $\al$ the coefficient of $(2k)^\al$ in this polynomial will be bounded in absolute value by 
	$2^{2k}(2n)^{2k-1-\al}$, so summing over all the terms involving $2k$ in the numerator gives, in absolute value, at most
	\[2^{2k}\sum_{\al=1}^{2k}(2k)^\al(2n)^{2k-1-\al}\le 2^{2k}(2k)^{2k+1} n^{2k-2}=O(n^{2k-2}),\]
	where we used that $k$ is a fixed constant. Dividing this by the denominator $(2n)^{2k}$ gives the result.
	
	Define
	\[f_{k,n}(t)=\sum_{q=0}^{2k-1} \sum_{\substack{\sig\in \S_{2,k},\\ \sig_1^{-1}=q+1}} F_k(\sig) \phi_{q,n}(t),\]
	and similarly define $g_{k,n}(t)$ but with $G_k(\sig)$ used instead of $F_k(\sig)$.  Then \eqref{E-Gam} and \eqref{E-Phi} imply 
	\[\E[S(\c{F}_{k,\gam},\rand{\pi})]\ge \sum_{t<2\gam n} f_{k,n}(t)+\sum_{t\ge 2\gam n} g_{k,n}(t)+O(n^{-1}).\]

	We wish to replace these sums with integrals using Lemma~\ref{lem:integral}.  Observe that $|f_{k,n}(t)|\le \f{k2^k}{n} |\S_{2,k}|$ for $0\le t\le 2n$ since $|\phi_{q,n}(t)|\le \f{2^{k-1}}{n}$ in this range and $F_k(\sig)\le 2k$ for all $\sig\in \S_{2,k}$.  Further, $f_{k,n}(t)$ has degree at most $2k$, so we can break $\R$ into at most $2k$ intervals  on which $f$ is monotone.  Similar analysis holds for $g_{k,n}(t)$.  By taking $r=2k=O(1)$ and $M=\f{2k}{n}|\S_{2,k}|=O(n^{-1})$, we conclude that 
	\begin{align*}\E[S(\c{F}_{k,\gam},\rand{\pi})]&\ge \int_1^{\floor{2\gam n}-1} f_{k,n}(t)\,dt+\int_{\ceil{2\gam n}}^{2n-1} g_{k,n}(t)\,dt+O(n^{-1})\\ &=\int_0^{2\gam n} f_{k,n}(t)\,dt+\int_{2\gam n}^{2n} g_{k,n}(t)\,dt+O(n^{-1}),\end{align*}
	where we used that $|g_{k,n}(t)|,|f_{k,n}(t)|= O(n^{-1})$ in this range to tweak the limits of integration.
	
	At this point one simply needs to choose some $k$ for which the polynomials $f_{k,n}$ and $g_{k,n}$ are feasible to compute and then to choose $\gam$ so that the corresponding integral is optimized.  For example, taking $k=2$ and using \eqref{E-Ex} gives that $f_{2,n}(t)$ equals
	{\small \[(2+4+3)\f{1}{3n}(1-t/2n)^{3}+(2+3)\f{1}{n}(t/2n)(1-t/2n)^2+1\cdot \f{1}{n}(t/2n)^2(1-t/2n),\]}
	and similarly $g_{2,n}(t)$ equals {\small\[(4+3+2)\f{1}{3n}(1-t/2n)^{3}+(3+2)\f{1}{n}(t/2n)(1-t/2n)^2+2\cdot \f{1}{n}(t/2n)^2(1-t/2n).\]}

	In this case $g_{2,n}(t)\ge f_{2,n}(t)$ for all $0\le t\le2n$, so the optimal choice is $\gam=0$, and evaluating $\int_0^{2n} g_{2,n}(t)\,dt$ gives the bound \[ \E[S(\c{F}_{2,0},\rand{\pi})]\ge 8/3+O(n^{-1}).\]
	
	In a similar fashion, it is feasible to compute the polynomials $f_{6,n}(t)$ and $g_{6,n}(t)$ by use of a computer.  Namely, one does this by computing $F_6(\pi)$ and $G_6(\pi)$ for each $\pi\in \S_{2,6}$, which took about a day using a laptop.  After recording the functions $f_{6,n}(t),g_{6,n}(t)$, one can choose various values of $\gam$ to plug into $\int_0^{2\gam n}f_{6,n}(t)\,dt+\int_{2\gam n}^{2n}g_{6,n}(t)\,dt$ in order to find a value of $\gam$ which gives a relatively large value.  Experimentally we found that $\gam=.35$ makes this integral fairly large, namely at least 2.9144.  In total this gives the bound \[ \E[S(\c{F}_{6,.35},\rand{\pi})]\ge 2.9143+O(n^{-1}),\]
	so for $n$ sufficiently large we obtain our desired bound.
	
	For larger $m$ a similar proof works, the only significant change being that we use
	\[\phi_{q,n}(x)=\f{(m!)^{k}}{(mn)q!(mk-1-q)!}(t/mn)^q(1-t/mn)^{mk-1-q}.\]
	For $m=3$, if we take $k=5$ and $\gam=.25$ (which again is found through computer computation and experimentation), this method ends up giving the desired lower bound for $\E[\b{S}_{3,n}]$ when $n$ is sufficiently large.
\end{proof}
In principle, many other strategies that utilize finite cutoffs could be analyzed with this method, which could further improve our asymptotic bounds.  However, it seems impractical to extend this method to larger $m$.

%A similar argument can be used to bound a strategy analogous to the $(\gam,\gam')$-shifting strategy to further improve the lower bound of Theorem~\ref{THM:LOWER}, and in principle many other strategies of this flavor could be analyzed using this same technique of introducing a finite cutoff value in order to prove asymptotic bounds. 

\section{Proof of Theorem~\ref{thm:minus}: the Avoiding Strategy $\c{A}$}\label{sec:minus}
This section does not appear in the journal version of this paper and has not been peer reviewed.

We recall that the avoiding strategy $\c{A}$ was defined in Section~\ref{sec:prac} as the strategy of guessing `1', then `2', and so on until some `$k$' is guessed correctly, and then `$k$' is guessed the rest of the game.  To determine $\E[S(\c{A},\rand{\pi})]$, we recall the Bonferonni inequalities, which are essentially weaker versions of the principle of inclusion-exclusion.

\begin{lem}[\cite{S}]
	Let $A_1,\ldots,A_{k-1}$ be events in a finite probability space.  For all even $L$ we have
	\[\Pr\l[\bigcap \ol{A_t}\r]\le \sum_{\ell=0}^L (-1)^\ell\sum_{1\le t_1<\cdots<t_\ell< k}\Pr\l[\bigcap_{j=1}^\ell A_{t_j}\r],\]
	and for all odd $L$ we have
	\[\Pr\l[\bigcap \ol{A_t}\r]\ge \sum_{\ell=0}^L (-1)^\ell\sum_{1\le t_1<\cdots<t_\ell< k}\Pr\l[\bigcap_{j=1}^\ell A_{t_j}\r].\]
\end{lem}
We use this to prove our main technical lemma of this section.
\begin{lem}\label{lem:approx}
	For $\pi\in \S_{m,n}$, let $f(\pi)$ denote the first index which is guessed correctly using the avoiding strategy $\c{A}$, with $f(\pi):=\infty$ if no such index exists.  For $m$ fixed and any $\ep>0$, we have for all $k\le mn$,
	\[\l|\Pr[f(\rand{\pi})=k]-e^{-k/n}n^{-1}\r|\le \ep n^{-1}+O_{\ep,m}(n^{-2}),\]
	and moreover
	{\small \[\l|\Pr[f(\rand{\pi})=k]\cdot \E[S(\c{A},\rand{\pi})|f(\rand{\pi})=k]-\l(m-\f{(m-1)k}{mn}\r)e^{-k/n}n^{-1}\r|\le \ep n^{-1}+O_{\ep,m}(n^{-2}).\]}
\end{lem}
\begin{proof}
	For simplicity we first prove these results for $k\le n$.  For $t\le n$ let $A_t$ be the event $\rand{\pi}_t=t$.  Also define $A_{*}$ to be the event $\rand{\pi}_{k+1}=k$.  Observe that $f(\rand{\pi})=k$ if and only if $A_k$ occurs and no other $A_t$ with $t<k$ occurs.  Given $f(\rand{\pi})=k$, the expected score is 1 plus the expected number of $k$'s after index $k$, which by linearity of expectation is 1 plus $mn-k$ times the probability of $\rand{\pi}_{k+1}=k$.  In total then we conclude 
	\begin{equation}\Pr[f(\rand{\pi})=k]=\Pr\l[\bigcap_{t<k}\ol{A_t}\cap A_k\r],\label{eq:avoidProb}\end{equation}
	and similarly
	{\small\begin{align}
		\Pr[f(\rand{\pi})=k]\cdot \E[S(\c{A},\rand{\pi})]&=\Pr\l[\bigcap_{t<k}\ol{A_t}\cap A_k\r]\cdot \l(1+(mn-k)\Pr\l[A_{*}|\bigcap_{t<k}A_t\cap A_k\r]\r)\nonumber \\&=\Pr\l[\bigcap_{t<k}\ol{A_t}\cap A_k\r]+(mn-k)\Pr\l[\bigcap_{t<k}\ol{A_t}\cap A_k\cap A_{*}\r]\label{eq:avoidEx}.\end{align}}
	
	We wish to bound these quantities using the Bonferonni inequalities. We claim for any $1\le t_1<\cdots<t_\ell\le k\le n$ that
	\[\Pr[A_{t_1}\cap \cdots \cap A_{t_\ell}]=\f{(mn-\ell)!m^\ell}{(mn)!}.\]
	Indeed, every $\pi\in \S_{m,n}$ with the property $\pi_{t_j}=t_j$ for $1\le j\le \ell$ can be formed by setting $\pi_{t_j}=t_j$ (with the $t_j$ values all distinct since $k\le n$) and then the remaining positions can be filled in $(mn-\ell)!(m!)^{-n+\ell}((m-1)!)^{-\ell}$ ways.  The total number of $\pi\in \S_{m,n}$ is $(mn)!(m!)^{-n}$, so dividing these two quantities gives the claim.  Using $(mn)^\ell\ge \f{(mn)!}{(mn-\ell)!}\ge (mn-\ell+1)^\ell$, we in particular find
	\begin{equation}\label{eq:minus}n^{-\ell}\le \Pr[A_{t_1}\cap \cdots \cap A_{t_\ell}\cap A_k]\le n^{-\ell-1}+O_\ell(n^{-\ell-2}).\end{equation}
	Similarly we find
	\begin{align}
		\Pr\l[\bigcap A_{t_j}\cap A_k\cap A_*\r]&=\f{(mn-\ell-2)!(m!)^{-n+\ell+1}((m-1)!)^{-\ell} (m-2)!^{-1}}{(mn)!(m!)^{-n}}\nonumber\\&=\f{m-1}{m} n^{-\ell-2}+\Theta_\ell(n^{-\ell-3}).\label{eq:minus2}
	\end{align}
	
	Using the Bonferonni inequality and \eqref{eq:minus} gives for any even $L\le k$ that 
	\begin{align}\Pr\l[\bigcap_{t<k}\ol{A_t}\cap A_k\r]&\le \sum_{\ell=0}^{L}(-1)^\ell\sum_{1\le t_1<\cdots<t_\ell<k}\Pr[\bigcap A_{t_j}\cap A_k]\nonumber\\&=n^{-1}\sum_{\ell=0}^L \f{(-k/n)^\ell}{\ell!}+O_\ell(k^\ell n^{-\ell-1})\nonumber\\&=n^{-1} e^{-k/n}+O(L^{-L}n^{-1})+O_L(n^{-2}),\label{eq:sum1}\end{align}
	where this last step used $|e^x-\sum_{\ell=0}^L\f{x^\ell}{\ell!}|\le \f{e^x}{(L+1)!}=O(L^{-L})$ for $x\le 1$.
	Similarly one can use \eqref{eq:minus2} to show
	{\small\begin{equation}\label{eq:sum2}(mn-k)\Pr\l[\bigcap_{t<k}\ol{A_t}\cap A_k\cap A_{*}\r]=\l(1-\f{k}{mn}\r)(m-1)n^{-1}e^{-k/n}+O(L^{-L}n^{-1})+O_L(n^{-2}),\end{equation}}
	and using odd $L$ gives the reveres inequalities.  We conclude the first result for $k\le n$ by using \eqref{eq:avoidProb} and \eqref{eq:sum1} after taking $L$ sufficiently large so that $O(L^{-L}n^{-1})\le \ep n^{-1}$.  Similarly we conclude the second result for $k\le n$ by using \eqref{eq:avoidEx}, \eqref{eq:sum1}, \eqref{eq:sum2}, and taking $L$ to be sufficiently large.
	
	We now consider the general case.  Given $t\le mn$, let $1\le t^{(1)}\le n$ and $0\le t^{(2)}<m$ be the unique integers such that $t=t^{(1)}+t^{(2)}n$. Let $A_t$ be the event $\rand{\pi}_{t}=t^{(1)}$ and $A_*$ the event $\rand{\pi}_{k+1}=k^{(1)}$.  With this notation established, the exact same argument as above will work,  with the only issue being that the bounds of  \eqref{eq:minus} and \eqref{eq:minus2} will be incorrect whenever we have two entries in our sequence with the same $t^{(1)}$ value.  However, the number of such sequences is $O(k^{\ell-1})$, and one can show that the probabilities are still $\Theta_\ell(n^{-\ell-1})$ for any sequence.  Thus these terms get absorbed in the error term and the bound continues to hold.
\end{proof}

With this we can determine the asymptotic value of $\E[S(\c{A},\rand{\pi})]$.

\begin{proof}[Proof of Proposition~\ref{prop:A}]
	Note that \[\E[S(\c{A},\rand{\pi})]=\sum \Pr[f(\rand{\pi})=k]\E[S(\c{A},\rand{\pi})|f(\rand{\pi})=k],\] where this conditional expectation is 0 for $k=\infty$.  Using this, Lemma~\ref{lem:approx}, and Lemma~\ref{lem:integral}, we find for all $\ep>0$
	\begin{align*}\E[S(\c{A},\rand{\pi})]&\le n^{-1}\sum_{k=1}^{mn}\l(m-\f{(m-1)k}{mn}\r)e^{-k/n}+\ep+O_{\ep,m}(n^{-1})\\&=\int_0^m \l(m-\f{(m-1)x}{m}\r)e^{-x}dx+m\ep+O_{\ep,m}(n^{-1})\\&=m-1+m^{-1}-m^{-1}e^{-m}+m\ep+O_{\ep,m}(n^{-1}).\end{align*}
	A similar lower bound holds with $\ep$ replaced by $-\ep$, and by taking $\ep$ arbitrarily small we conclude the desired result.
\end{proof}

The last fact we need to prove Theorem~\ref{thm:minus} is the following.
\begin{lem}\label{lem:X}
	Given a strategy $\c{G}$ with Yes/No feedback, let $X(\c{G})$ be the event that $S(\c{G},\rand{\pi})=0$.  Then $\Pr[X(\c{A})]\ge \Pr[X(\c{G})]$ for all strategies $\c{G}$.
\end{lem}
\begin{proof}
	Let $\c{G}$ be any strategy, and let $g_t$ denote the guess made for $\rand{\pi}_t$ assuming no correct guesses have been made up to this point.  Let $c_i=|\{t:g_t=i\}|$ and let $\S_{m,n}(c_1,\ldots,c_n)$ denote the number of permutations which have no 1 in the first $c_1$ positions, then no 2 in the next $c_2$ positions, and so on.  It is not too difficult to see that $\Pr[X(\c{G})]=|\S_{m,n}(c_1,\ldots,c_n)|/|\S_{m,n}|$.  Intuitively the quantity $|\S_{m,n}(c_1,\ldots,c_n)|$ is maximized when $c_i=m$ for all $i$, and one can formally show this by using Theorem~4 of \cite{CDGM}.  
	
	In total, $\Pr[X(\c{G})]$ is maximized by any strategy which guesses each card type $m$ times whenever it guesses no cards correctly.  The result follows since $\c{A}$ is such a strategy.
\end{proof}

\begin{proof}[Proof of Theorem~\ref{thm:minus}]
	The upper bound follows from Proposition~\ref{prop:A}.  For the lower bound, fix some strategy $\c{G}$.  Observe that $\E[S(\c{G},\rand{\pi})]$ is always at least the probability that some card is guessed correctly using $\c{G}$, so by Lemma~\ref{lem:X} this quantity is at least $1-\Pr[X(\c{A})]$.  In the language of Lemma~\ref{lem:approx}, this is equal to $\Pr[f(\rand{\pi})\le mn]$.  By Lemmas~\ref{lem:approx} and \ref{lem:integral}, this quantity is at least
	\[\int_0^m e^{-x}dx-\ep m-O_{\ep,m}(n^{-1})=1-e^{-m}-\ep m-O_{\ep,m}(n^{-1}).\]
	Taking $\ep$ arbitrarily small gives the desired result.
\end{proof}

\section{Concluding Remarks.}\label{sec:concluding}
There are many questions that remain in this area.  Perhaps the most embarrassing gap in our knowledge is in understanding the score under the optimal mis\`ere strategy with Yes/No feedback.  Theorem~\ref{thm:minus} proves that this quantity is bounded away from 0, but there is still a large gap between this and the upper bound of Theorem~\ref{THM:LOWER}.  We conjecture the following.

\begin{conj}
	If $n$ is sufficiently large in terms of $m$, then under an optimal mis\`ere strategy with Yes/No feedback,
	\[\E[\b{S}_{m,n}]=m-o(m),\]
	where this bound holds uniformly in $n$.
\end{conj}
Observe that Theorem~\ref{thm:partAsy} shows that the analogous conjecture for the optimal strategy is true.

It would be of significant interest to determine the asymptotic value of $\E[\b{S}_{2,n}]$ under an optimal strategy with Yes/No feedback, but this seems difficult.  We claim that one can adapt the methods of \cite{DGHS} to prove $\E[\b{S}_{m,n}]\le 2.7m$ for any $m$ provided $n$ is sufficiently large in terms of $m$.  In light of Theorem~\ref{thm:partAsy}, we know that the ratio $\E[\b{S}_{m,n}]/m$ tends to 1 as a function of $m$ provided $n$ is sufficiently large in terms of $m$.  Given this and Theorem~\ref{thm:1}, we suspect the following holds.
\begin{conj}
	For all $m$ and $\ep>0$, if $n$ is sufficiently large in terms of $m$, we have under an optimal strategy with Yes/No feedback,
	\[\E[\b{S}_{m,n}]\le (e-1)m+\ep.\]
\end{conj}

Another problem of interest would be to close the gap in the error terms of Theorems~\ref{THM:LOWER} and \ref{thm:partAsy} for the optimal strategy.  In particular, we conjecture the following.
\begin{conj}
	For all $\ep>0$, if $n$ is sufficiently large in terms of $m$, then under an optimal strategy with Yes/No feedback,
	\[\E[\b{S}_{m,n}]=m+O(m^{1/2+\ep}),\]
	where this bound holds uniformly in $n$.
\end{conj}
That is, we suspect the lower bound of Theorem~\ref{THM:LOWER} to be essentially best possible.  %This is motivated in part because, based on heuristic and computational evidence; the halfway, safe, and shifting strategies all seem to give a score of $m+\Theta(m^{1/2})$.

%We end this article with a delightful conjecture regarding the nature of optimal strategies which was proposed in \cite{CDGM}: With Yes/No feedback, if for the $t$th guess, an optimal strategy says to ``Guess $i$'' and one receives ``No'' for feedback, then an optimal strategy says to `Guess $i$'' for the $t+1$st guess.

\begin{acknowledgment}{Acknowledgments}
The authors would like to thank Steve Butler, Maryanne Gatto, Xiaoyu He, and Matthew Kwan for fruitful discussions.  The authors would also like to thank the referees and editors for their helpful comments, and in particular for improving the statement and proof of Theorem 1.  The first author was partially supported by National Science Grant  DMS 195404.  The last author was partially supported by  the National Science Foundation Graduate Research Fellowship under Grant No. DGE-1650112.
\end{acknowledgment}

\begin{biog}
		\item[Persi Diaconis] has been shuffling cards since he was five years old. He learned to perfectly shuffle at age 13 and started proving theorems about shuffling at age 20 (he is still at it). He is professor of mathematics and statistics at Stanford University.
		
		\begin{affil}
			Department of Mathematics and Statistics, Stanford, Stanford CA 94305\\ 
				diaconis@math.stanford.edu
		\end{affil}
		
		\item[Ron Graham] was the Irwin and Joan Jacobs Professor in the Department of Computer Science and Engineering at the University of California, San Diego, and the Chief Scientist of the California Institute for Telecommunications and Information Technology. He was a former president of the Mathematical Association of America and the American Mathematical Society and had been awarded numerous honors for his extensive groundbreaking work in discrete mathematics. He co-wrote Magical Mathematics with Persi Diaconis, which was awarded the 2013 Euler Book Prize.
		
		\begin{affil}
			Department of Mathematics, UC San Diego, La Jolla CA 92093
		\end{affil}
		
		\item[Sam Spiro] is a PhD student at UC San Diego studying combinatorics with Jacques Verstra\"ete.  In addition to playing card games, he also enjoys reading books and making math puns.
		\begin{affil}
			Department of Mathematics, UC San Diego, La Jolla CA 92093\\
			sspiro@ucsd.edu
		\end{affil}
	\end{biog}
\end{document}